\documentclass[12pt,a4paper]{amsart}

\usepackage[T1]{fontenc}
\usepackage{geometry} 
\usepackage{amssymb}

\usepackage{verbatim,enumerate}
\usepackage{xcolor}

\usepackage[looser]{newpxtext}
\usepackage[smallerops]{newpxmath}
\usepackage{fancyhdr}
\makeatletter
\def\@@and{\MakeLowercase{and}}
\makeatother

\usepackage[hidelinks,pagebackref]{hyperref}
\renewcommand*{\backref}[1]{}
\renewcommand*{\backrefalt}[4]{%
    \ifcase #1 (Not cited.)%
    \or        (Cited on page~#2.)%
    \else      (Cited on pages~#2.)%
    \fi
    }

\theoremstyle{definition}
\newtheorem{defn}{Definition}[section]
\newtheorem{exam}[defn]{Example}

\newtheorem{rem}[defn]{Remark}

\theoremstyle{plain}
\newtheorem{thm}[defn]{Theorem}
\newtheorem{lem}[defn]{Lemma}
\newtheorem{prop}[defn]{Proposition}
\newtheorem{coro}[defn]{Corollary}

\newcommand{\eps}{\varepsilon}
\newcommand{\calt}{\mathcal{T}}

\newcommand{\bbn}{\mathbb{N}}
\newcommand{\bbk}{\mathbb{K}}
\newcommand{\bbr}{\mathbb{R}}
\newcommand{\bbc}{\mathbb{C}}

\newcommand{\vecz}{\mathbf{0}}
\newcommand{\dd}{\mathop{}\!\mathrm{d}}
\DeclareMathOperator{\sspan}{span}

\DeclareMathOperator{\udens}{\overline{\mathrm{dens}}}
\DeclareMathOperator{\dens}{\mathrm{dens}}
\DeclareMathOperator{\ldens}{\underline{\mathrm{dens}}}

\DeclareMathOperator{\asym}{Asym}

\DeclareMathOperator{\dasym}{DAsym}
\DeclareMathOperator{\dprox}{DProx}
\DeclareMathOperator{\dunbd}{DUnbd}

\title{D\MakeLowercase{istributionally chaotic} $C_0$-\MakeLowercase{semigroups on complex sectors}}	

\author[Z. J\MakeLowercase{iang} ]{Z\MakeLowercase{hen}  Jiang}

\address[Z. Jiang]{Department of Mathematics,
	Shantou University, Shantou, 515821, Guangdong, China}
\email{jiangzhen17@outlook.com}
\urladdr{https://orcid.org/0000-0001-8013-9317}

\author[J. L\MakeLowercase{i}]{J\MakeLowercase{ian} Li}
\address[J. Li]{Department of Mathematics,
	Shantou University, Shantou, 515821, Guangdong, China}
\email{lijian09@mail.ustc.edu.cn}
\urladdr{https://orcid.org/0000-0002-8724-3050}

\author[Y. Y\MakeLowercase{ang}]{Y\MakeLowercase{ini} Yang}
\address[Y. Yang]{Department of Mathematics,
	Shantou University, Shantou, 515821, Guangdong, China}
\email{ynyangchs@foxmail.com}
\urladdr{https://orcid.org/0000-0001-6564-2213}

	

\begin{document}

\begin{abstract}
We explore distributional chaos for $C_0$-semigroups of linear operators on Banach spaces whose index set is a sector in the complex plane. 
We establish the relationship between distributional sensitivity and distributional chaos by
characterizing them in terms of distributionally (semi-)irregular vectors.
Additionally, we provide conditions under which a $C_0$-semigroup admits a linear manifold of distributionally irregular vectors.
Furthermore, we delve into the study of  distributional chaos for the translation $C_0$-semigroup on weighted $L_p$-spaces with a complex sector as the index set.
We obtain a sufficient condition for dense distributional chaos, expressed in terms of the weight. 
In particular, we construct  an example of a translation $C_0$-semigroup with a complex sector index set that is Devaney chaotic but not distributionally chaotic.
\end{abstract}

\maketitle


\section{Introduction}
In the past twenty years, a lot of researches has focused on various chaotic behaviors of linear operators, such as Devaney chaos \cite{GP2011}, Li-Yorke chaos \cite{BBMP2011,BBMP2015}, distributional chaos \cite{BBMP2013,BBPW2018}, and mean Li-Yorke chaos \cite{BBP2020}. 
To study the asymptotic behavior of the solutions of some differential equations, Desch, Schappacher, and Webb carried out a systematic study of the chaotic properties of a C0-semigroup in \cite{DSW1997}.
Now many of these results have been extended to $C_0$-semigroups, which can be regarded as a continuous generalization compared with the iteration of a single operator.  
We refer the reader to the monographs \cite{GP2011,BM2009} for more details on linear systems and to the book \cite{EN2000} for $C_0$-semigroups.

Recall that a one-parameter family $\{T_t\}_{t\in \bbr^+}$ of continuous linear operators on a Banach space $X$ is a $C_0$-semigroup if $T_0=I$, $T_s \circ T_t=T_{s+t}$ for all $ s, t \geq 0$, and $\lim_{s\to t}T_{s}x=T_t x$,  for all $x\in X$ and $t\geq 0$. 
Researchers have established connections between the behavior of full orbits under the $C_0$-semigroup and the behavior of discrete orbits corresponding to the iterates of the operators in the semigroup.
In \cite{CMP2007}, Conejero et al. proved that every single operator in a hypercyclic $C_0$-semigroup is also hypercyclic. Similar results hold for the property of distributional chaos \cite{ABMP2013} and mean Li-Yorke chaos \cite{BBP2020}.  

Recently, researchers have discussed various kinds of the dynamical properties of the translation $C_0$-semigroups on complex sectors of the form $\Delta:=\Delta(\alpha)=\{r e^{i\theta}\colon r\geq 0, |\theta|\leq \alpha\}$, for some $\alpha \in (0,\pi/2)$, such as Devaney chaos~\cite{CP2007}, hypercyclicity~\cite{CP2009} and recurrent hypercyclicity criterion~\cite{LXZ2023}. 
In~\cite{YCN2017}, Yao et al. studied distributional chaotic translation $C_0$-semigroups on complex sectors. 
In~\cite{CKPV2020}, Chaouchi et al. discussed the $f$-frequently hypercyclicity
$C_0$-semigroups on complex sectors.  
In~\cite{HLYS2024}, He et al. investigated transitivity and Li-Yorke chaos for $C_0$-semigroups on complex sectors. 

For the $C_0$-semigroup on complex sectors $\{T_t\}_{t\in \Delta}$, the situation becomes more complicated. On the one hand, the $C_0$-semigroup $\{T_t\}_{t\in \Delta}$ and its non-trivial operators no longer exhibit similar dynamical properties. 
For example,
Conejero and Peris \cite{CP2009} constructed a hypercyclic $C_0$-semigroups $\{T_t\}_{t\in \Delta}$ with a non-hypercyclic operator $T_{t_0}$ for any $ t_0\in \Delta$. 
On the other hand, the dynamical behavior of the semigroup on a complex sector index set  is much richer.
It is well known that a continuous linear operator $T: X\to X$ must be sensitive if it 
is Li-Yorke chaotic (resp. distributional  sensitive if it is distributional chaotic). 
In fact, this result is also valid for the $C_0$-semigroup $\{T_t\}_{t\in \bbr^+}$. 
Surprisingly, He et al. \cite{HLYS2024} constructed a $C_0$-semigroup on a complex sector that is Li-Yorke chaotic but not sensitive.

Using techniques from topological dynamics, the authors in \cite{JL2025} characterize three types of chaos within the framework of continuous endomorphisms of completely metrizable groups. 
In this paper, we extend this investigation to explore distributional chaos for $C_0$-semigroups on complex sectors.
We provide a characterization of distributional chaos in terms of semi-irregular vectors (respectively, irregular vectors) for $C_0$-semigroups on complex sectors.
Unlike Li-Yorke chaos, we observe that a distributionally chaotic semigroup on a complex sector index set must exhibit distributional sensitivity.
The key insight is to demonstrate that the upper density of a measurable set in a complex sector remains invariant under translation (See Lemma \ref{transitive-invariant-property-of-density}). Furthermore, in \cite{BP2012} the authors showed that Devaney chaos implies distributional chaos for the translation semigroup with $\bbr^+$ as an index set. However, we present an example illustrating that Devaney chaos does not necessarily imply distributional chaos for $C_0$-semigroups on complex sectors (See Example~\ref{Devaney-chaos-not-distributionally-chaos}). Notably, the translation semigroup constructed in this example exhibits distributional chaos when restricted to a ray, which contrasts with the case of nonnegative real number indices (See ~\cite{ABMP2013} and ~\cite{BBPW2018}).

The paper is organized as follows:
In Section 2, we recall some notations and preliminaries.
In Section 3, we study distributional  sensitivity and distributional chaos for a $C_0$-semigroup with a complex sector index set. 
In particular, using the so-called distributional semi-irregular vectors and distributional irregular vectors, 
we characterize distributional sensitivity and distributional chaos,
and investigate their relationships
for the $C_0$-semigroup with a complex sector as an index set. 
In Section 4, we further study the translation $C_0$-semigroup on weighted $L_p$-spaces. 
We give a sufficient condition for dense distributional chaos expressed in terms of the weight.
In particular, we construct an example of a translation $C_0$-semigroup on a complex sector that is Devaney chaotic but not distributional chaotic. 

\section{Preliminaries}
Let $\bbn$ denote the set of all positive integers, $\bbr$ the set of all real numbers and $\mathbb{C}$ the complex plane.
A \emph{complex sector} in $\mathbb{C}$ is defined as 
$\Delta(\alpha)=\{r e^{i\theta}\in\mathbb{C}\colon r\geq 0, |\theta|\leq \alpha\}$, for some $\alpha \in (0,\pi/2)$. 
For every $r>0$, we define $\Delta_r(\alpha)=\{t\in\Delta(\alpha)\colon |t| <r\}$.
It is clear that every complex sector $\Delta(\alpha)$ forms a semigroup under the operation of addition.
When the parameter $\alpha$ is unambiguous, we will denote  $\Delta(\alpha)$ simply by $\Delta$ for brevity.

Let $A$ be a Lebesgue measurable subset  of a complex sector $\Delta$.
The \emph{Lebesgue measure} of $A$ is denoted by $\mu(A)$.
The \emph{upper density} of $A$ (with respect to $\Delta$) is defined as
\[
  \udens(A)=\limsup_{r\to\infty}
  \frac{\mu(A\cap \Delta_r)}{\mu(\Delta_r)},
\]
and the \emph{lower density} of $A$ is defined as
\[      
  \ldens(A)=\liminf_{r\to\infty}
  \frac{\mu(A\cap \Delta_r)}{\mu(\Delta_r)}.
\]
If $\udens(A)=\ldens(A)$, the common value is called the \emph{density} of $A$, denoted by $\dens(A)$.

For $A\subset \Delta$ and $t_0\in\Delta$,
let $A-t_0=\{s\in\Delta \colon s+t_0\in A\}$
and $A+t_0=\{s\in\Delta \colon s-t_0\in A\}$.
We will need the following key lemma, which shows that
the upper density of a measurable set is invariant under translation.

\begin{lem}\label{transitive-invariant-property-of-density}
  Let $A\subset \Delta$ be a measurable subset
  and $t_0\in\Delta$. Then 
  \[
    \udens(A)=\udens(A-t_0)=\udens(A+t_0).
  \]
\end{lem}
\begin{proof}
Since $t_0\in\Delta$, we denote $t_0=r_0 e^{i\theta_0}$
where $r_0>0$ and $|\theta_0|<\alpha$. First we show that $\udens(A)=\udens(A-t_0)$.
  For any $r>0$, $t_0+\Delta_r\subset \Delta_{r+r_0}$. Thus we have
  \begin{align*}
  (A\cap \Delta_{r+r_0})\setminus (\Delta_{r+r_0}\setminus (t_0+\Delta_r)) -t_0 & =A\cap \Delta_{r+r_0}\cap(t_0+\Delta_r)-t_0 \\
  &\subset A\cap(t_0+\Delta_r)-t_0=(A-t_0)\cap\Delta_r\\
  &\subset A\cap \Delta_{r+r_0}-t_0,
  \end{align*}
  and
  $\mu(A\cap \Delta_{r+r_0})-\mu(\Delta_{r+r_0})+\mu(\Delta_r)\leq
    \mu((A-t_0)\cap \Delta_r)\leq
    \mu(A\cap \Delta_{r+r_0})$.
  Then
  \begin{align*}
    \udens(A-t_0) & =\limsup_{r\to\infty}
    \frac{\mu((A-t_0)\cap \Delta_r)}{\mu(\Delta_r)} \\
               & \geq \limsup_{r\to\infty}  \frac{\mu(\Delta_{r+r_0})}{\mu(\Delta_r)} \frac{1}{\mu(\Delta_{r+r_0})} (\mu(A\cap \Delta_{r+r_0})-\mu(\Delta_{r+r_0})+\mu(\Delta_r)) \\
                & =\limsup_{r\to\infty} \frac{\mu(A\cap \Delta_{r+r_0})}{\mu(\Delta_{r+r_0})} =\udens(A)
  \end{align*}
  and
  \begin{align*}
    \udens(A-t_0) & =\limsup_{r\to\infty}
    \frac{\mu((A-t_0)\cap \Delta_r)}{\mu(\Delta_r)} \\
                & \leq
    \limsup_{r\to\infty}  \frac{\mu(\Delta_{r+r_0})}{\mu(\Delta_r)} \frac{\mu(A\cap \Delta_{r+r_0})}{\mu(\Delta_{r+r_0})}=\udens(A).
  \end{align*}
  This implies that  $\udens(A-t_0)=\udens(A)$.

  Since $(A+t_0)-t_0=A$, $\udens(A+t_0)=\udens(A)$.
\end{proof}

Let $X$ be a topological space. A subset of $X$ is called a \emph{$G_{\delta}$ set} if it can be expressed as the countable intersection of open subsets of $X$, a \emph{Cantor set} if it is homeomorphic to the Cantor ternary set in the unit interval. 
In a completely metrizable space $X$, a subset
of $X$ is called \emph{residual} if it contains a dense $G_\delta$ subset of $X$.
The following classical result is due to Mycielski~\cite{M1964}, which is an important topological tool to construct a chaotic set.

\begin{thm}\label{Mycielski}
  Let $X$ be a completely metrizable space without isolated points and $R$ be a dense $G_{\delta}$ subset of $X\times X$.
  Then for every sequence $(O_j)_j$ of nonempty open subsets of $X$, there exists a sequence $(K_j)_j$ of Cantor sets with $K_j\subset O_j$ such that
  \[
    \bigcup_{j=1}^\infty   K_j\times  \bigcup_{j=1}^\infty K_j\subset R\cup \Delta_X,
  \]
  where $\Delta_X=\{(x_1,x_2)\in X\times X\colon  x_1=x_2\}$.
  In addition, if $X$ is separable then we can require that the $\sigma$-Cantor set $\bigcup_{j=1}^\infty K_j$ is dense in $X$.
\end{thm}

Let $X$ be a Banach space and $L(X)$ be the set of all continuous linear operators $T: X\to X$. 
A one-parameter family of operators $\{T_t\}_{t\in\Delta}\subset L(X)$ is a \emph{semigroup} if $T_0=I$,
and $T_t T_s=T_{s+t}$ for all $t,s\in\Delta$.
A semigroup $\{T_t\}_{t\in\Delta}$ is said to be a \emph{strongly continuous semigroup} or a \emph{$C_0$-semigroup} if for all $s\in\Delta$ we have $\lim_{t\to s} T_t x=T_s x$ in the norm for all $x\in X$.
By the Banach-Steinhaus Theorem, the $C_0$-semigroup $\{T_t\}_{t\in\Delta}$ is locally equicontinuous, i.e., for every $r>0$ we have $\sup\{\|T_t\| \colon t\in \Delta_r\}<\infty$. 
Moreover, similar to the case of nonnegative real number indices, the operator norm of a $C_0$-semigroup is exponentially bounded. 
The following propositions generalize \cite[Lemma~7.1]{GP2011} and \cite[Propposition~7.3]{GP2011}, whose proofs are immediate.

\begin{prop}\label{continuous-of-C0-semigroups}
Let $\calt=\{T_t\}_{t \in \Delta}$ be a family of operators on a Banach space $X$. 
Then the following assertions are equivalent:
  \begin{enumerate}
    \item $\lim_{s\to t}T_{s}x=T_t x$,  for all $x\in X$ and $t\geq 0$;
    \item $\{T_t\}_{t \in \Delta}$ is locally equicontinuous and there is a dense subset $X_0$ of $X$ such that 
    \[
     \lim_{s\to t}T_{s}x=T_t x, 
    \]
    for all $x\in X_0$ and $t\geq 0$;
    \item the map
          \[
            \Delta \times X\to X, (t, x)\to T_tx
          \]
          is continuous.
  \end{enumerate}
\end{prop}

\begin{prop}
Let $\calt=\{T_t\}_{t \in \Delta}$ be a $C_0$-semigroup of operators on a Banach space $X$.
Then there exist $M\geq 1$ and $w\in \bbr$ such that  $\|T_t\|\leq M e^{w|t|}$ for all $t\in \Delta$.
\end{prop}

\section{Distributional chaos for \texorpdfstring{$C_0$}{C0}-semigroups on complex sectors}
In this section, we will discuss the distributional chaos for the $C_0$-semigroup defined on a complex sector. 

\subsection{Distributional asymptoticicity and distributional  proximality} 
In this subsection, we study the distributionally asymptotic relation and the distributionally proximal relation.

\begin{defn}
Let $\calt=\{T_t\}_{t \in \Delta}$ be a $C_0$-semigroup on $X$.
A pair $(x,y)\in X\times X$ is called \emph{distributionally asymptotic} if for every $\eps>0$,
\[
  \dens(\{t\in \Delta: \|T_tx-T_ty\|<\eps\})=1,
\]
and \emph{distributionally proximal} if for every $\eps>0$,
\[
  \udens(\{t\in \Delta: \|T_tx-T_ty\|< \eps\})=1.
\]
The \emph{distributionally asymptotic relation} and the \emph{distributionally proximal relation} of $\calt$, denoted by $\dasym(\calt)$ and $\dprox(\calt)$ , are the set of all distributionally asymptotic pairs and  distributionally proximal pairs respectively.

For any $x\in X$, the \emph{distributionally asymptotic} and the \emph{distributionally proximal} cell of $x$ are defined by
\[
  \dasym(\calt,x)=\{y\in X\colon (x,y)\in \dasym(\calt)\}
\]
and
\[
  \dprox(\calt,x)=\{y\in X \colon (x,y) \in \dprox(\calt)\},
\]
respectively.
\end{defn}

\begin{rem}
It is clear that $\dasym(\calt)\subset \dprox(\calt)$ 
and $\dasym(\calt,x)\subset \dprox(\calt,x)$ for every $x\in X$.
\end{rem}

\begin{lem}\label{G_delta-set}
  Let $\calt=\{T_t\}_{t \in \Delta}$ be a $C_0$-semigroup on $X$. Then for every $x\in X$, $\dprox(\calt,x)$ is a $G_\delta$ subset of $X$. Similarly, $\dprox(\calt)$ is a $G_\delta$ subset of $X\times X$.    
\end{lem}
\begin{proof}
  We will show that $\dprox(\calt,x)$ is a $G_\delta$ subset of $X$. The 
  proof of the case of $\dprox(\calt)$ is similar.
  
  For any $x\in X$ and $n\in \bbn$, denote
  \[
    X_n= \bigl\{ y\in X: \exists r>n  \text{ s.t. } \mu(\{t\in \Delta_r: \|T_t y-T_t x\|<\tfrac{1}{n}\})>\mu(\Delta_r)(1-\tfrac{1}{n})\bigr\}.
  \]
  Then $\dprox(\calt, x)=\bigcap_{n=1}^{\infty}X_n$. We will show that $X_n$ is open for every $n\in\bbn$.
  
  Fix $n\in\bbn$ and $y\in X_n$, then there exists $r_1>n$ such that
  \[
    \mu(\{t\in \Delta_{r_1}: \|T_t y-T_t{x}\|<\tfrac{1}{n}\})>\mu(\Delta_{r_1})(1-\tfrac{1}{n}).
  \]
  By the regular property of the Lebesgue measures, there exists some close set $K$ such that
  \[
    K\subset \{t\in \Delta_{r_1}: \|T_t y-T_t{x}\|<\tfrac{1}{n}\}
  \]
  and
  \[
    \mu(K)>\mu(\Delta_{r_1})(1-\tfrac{1}{n}).
  \]
  Since the $C_0$-semigroup $\{T_t\}_{t\in\Delta}$ is locally equicontinuous,  $M:=\sup_{t\in \Delta_{r_1}}\{\|T_t\|\}<\infty$. 
  By Proposition \ref{continuous-of-C0-semigroups}, the map $\Delta \times X\to X, (t, x)\to T_tx$ 
  is continuous.
  Let $\delta:=\frac{1}{n}-\max_{t\in K}\|T_{t}y-T_{t}x\|>0$.
  For any $z\in B_{\|\cdot\|}(y,\frac{\delta}{2M})$ and $t\in K$
  we have
  \begin{equation*}
    \begin{split}
 \|T_t z-T_t{x}\|& =  \|T_t z- T_t y+T_t y-T_tx\|      \\
      &\leq     \|T_t\|\|z-y\|+\|T_t y-T_tx\| \\
     & \leq M\cdot\frac{\delta}{2M}+\|T_ty-T_tx\|<\frac{1}{n}.
    \end{split}
  \end{equation*}
Then $\{t\in \Delta_{r_1}: \|T_t z-T_t{x}\|<\frac{1}{n}\}\supset K$. Consequently,
 \[
 \mu(\{t\in \Delta_{r_1}: \|T_t z-T_t{x}\|<\tfrac{1}{n}\})\geq \mu(K)>\mu(\Delta_{r_1})(1-\tfrac{1}{n}), 
 \] 
which means that $z\in X_n$. Thus, $X_n$ is an open set.  
\end{proof}

\subsection{Distributional sensitivity}
In this subsection we investigate distributional sensitivity.

\begin{defn}
  Let $\calt=\{T_t\}_{t \in \Delta}$ be a $C_0$-semigroup on $X$.
  The semigroup  $\calt$ is said to be \emph{distributionally sensitive} if there exists $\delta>0$ such that for any $x\in X$ and $\eps>0$, there exists $y\in X$ with $\|x-y\|<\eps$ such that
  \[
    \udens(\{t\in \Delta: \|T_tx-T_ty\|> \delta)=1.
  \]
  The number $\delta$ is called a distributionally  sensitivity constant for $\{T_t\}_{t \in \Delta}$.  
\end{defn}

Now we introduce distributionally unbounded vectors and characterize distributional sensitivity. 
\begin{defn}
Let $\calt=\{T_t\}_{t \in \Delta}$ be a $C_0$-semigroup on $X$. A vector $x\in X$ is called \emph{distributionally unbounded} if there exists a measurable subset $A$ of $\Delta$ with $\udens(A)=1$ such that
  \[
    \lim_{t\in A, \, |t|\to \infty} \|T_t x\|=\infty.
  \]
Denote the collection of all distributionally unbounded vectors as $\dunbd(\calt)$.
\end{defn}

\begin{prop}\label{characterized-for-distributionally-sensitive}
Let $\calt=\{T_t\}_{t \in \Delta}$ be a $C_0$-semigroup on $X$. 
Then the following assertions are equivalent:
  \begin{enumerate}
    \item $\calt$ is distributionally sensitive; 
    \item $\calt$ admits a distributionally unbounded vector;
    \item $\dunbd(\calt)$ is a dense $G_{\delta}$ subset of $X$.
  \end{enumerate}
\end{prop}
\begin{proof}
$(1)\Rightarrow(3)$. For any $n\in \bbn$, denote
  \[
    X_n:=\Bigl\{x\in X: \exists r>n \quad \text{s.t.} \quad \mu(\{t\in \Delta_r: \|T_t x\|>n\})> \mu(\Delta_r)(1-\frac{1}{n})  \Bigr\}.
  \]
  Then $
  \dunbd(\calt)=\bigcap_{n=1}^{\infty} X_n$.
  Similar to the proof of Lemma~\ref{G_delta-set}, we know that $X_n$ is open for every $n\in\bbn$. 
  
  Fix $n\in\bbn$, 
  we now show that $X_n$ is dense. 
  For any nonempty open set $U$, there exists $x_0\in X$ and $\eps>0$ such that
  $B_{\|\cdot\|}(x_0,\eps)\subset U$.
   Let $\delta$ be the distributionally sensitive constant. 
   There exists $x'\in X$ with $\|x'\|\leq M:=\frac{\eps\delta}{4n^2}$ and ${r_1}>n$ such that
  \[
    \mu(\{t\in \Delta_{r_1}: \|T^i x'\|>\delta\})\geq \mu(\Delta_{r_1})\bigl(1-\tfrac{1}{2n}\bigr).
  \]
  Consider the vectors
  \[
    x_s:=x_0+\frac{\eps s x'}{2nM} \quad (s=0,1,...2n-1).
  \]
  We find that
  \[
    \|x_s- x_0\|=\Bigl\|\frac{\eps s x'}{2nM}\Bigr\|\leq \eps.
  \]
  Then we have that $\{x_s: 0\leq s\leq 2n-1\}\subset U$. We show that there exists some $0\leq s_0\leq 2n-1$ such that $x_{s_0}\in X_n$.
  Let $A:=\{t\in \Delta_{r_1}: \|T_t x'\|\}>\delta\}.$ Then $\mu(A)\geq \mu(\Delta_{r_1})(1-\frac{1}{2n})$. Let $B_s:=\{j\in \Delta_r: \|T_j x_s\|\leq n\}$. If $s, m\in \{0,1,...,2n-1\}$ and $s\neq m$, then $B_s \cap B_m \cap A=\varnothing$. If not, there exists some $s\neq m$ and $t_1\in B_s \cap B_m\cap A$. Then
  \[
    \|T_{t_1}x_s-T_{t_1}x_m\|= \frac{\|s-m\|\cdot \eps \|T_{t_1} x'\|}{2nM}> \frac{\eps\delta}{2nM}=2n
  \]
  and
  \[
    \|T_{t_1}x_s-T_{t_1}x_m\|\leq \|T_{t_1}x_s\|+\|T_{t_1}x_m\|\leq 2n,
  \]
  which is a contradiction. So there exists some $0\leq s_0\leq 2n-1$ such that $\mu(B_{s_0}\cap A)\leq \frac{\mu(A)}{2n}$.
  Then
  \[
    \mu(A \setminus B_{s_0})\geq \mu(\Delta_{r_1})\bigl(1-\tfrac{1}{2n}\bigr)^2\geq \mu(\Delta_{r_1}) \bigl(1-\tfrac{1}{n}\bigr).
  \]
  For any $t\in A \setminus B_{s_0}$ we have $\|T_t x_{s_0}\|> n$, which implies that $ x_{s_0}\in X_n$, as desired.

  $(3)\Rightarrow(2)$ is trivial.

  $(2)\Rightarrow(1)$. Let $x_0$ be a distributionally unbounded vector. For any $x\in X$ and $\eps>0$, there exists some $k_0\in \bbk\setminus\{0\}$ such that $\|k_0x_0\|<\eps$. Clearly, $k_0 x_0$ is also distributionally unbounded. For any $M>0$, we have $\|x+k_0x_0-x\|<\eps$
  and
  \[
    \udens(\{t\in\Delta \colon \|T_t(x)-T_t(x+k_0x_0)\|> M)=1.
  \]
  Hence, $\{T_t\}_{t \in \Delta}$ is distributionally sensitive.
\end{proof}

Now we will introduce the definition of distributionally semi-irregular vectors and 
distributionally irregular vectors. 
\begin{defn}
   Let $\calt=\{T_t\}_{t \in \Delta}$ be a $C_0$-semigroup on $X$.
  A vector $x\in X$ is called \emph{distributionally semi-irregular} if there exist measurable subsets $A,B\subset \Delta$ with $\udens(A)=1$ 
  and $\udens(B)=1$
  such that 
  \[\lim_{t\in A,\, |t|\to \infty}\|T_tx\|=0\ \text{and}\ \lim_{t\in B,\, |t|\to\infty}\|T_t x\|>0,\] 
and  \emph{distributionally irregular} if there exist measurable subsets $A,B\subset \Delta$ with $\udens(A)=1$ 
  and $\udens(B)=1$ such that 
  \[\lim_{t\in A,\, |t|\to \infty}\|T_tx\|=0\ \text{and}\
  \lim_{t\in B,\,|t|\to\infty}\|T_t x\|=\infty.\] 
\end{defn}

He et al. \cite{HLYS2024} observed that there exists a semigroup $\calt=\{T_{t}\}_{t\in \Delta}$ admits a semi-irregular vectors but not sensitive.
However, we will prove that if a semigroup $\calt=\{T_{t}\}_{t\in \Delta}$ admits a distributionally semi-irregular vector, then it is distributionally sensitive.

To this end, we need the following proposition, which shows that the set of all distributionally semi-irregular vectors is a $\calt$-invariant set. 

\begin{prop}\label{the-orbit-of-DC-vector-is-DC-vector}
  Let $\calt=\{T_t\}_{t \in \Delta}$ be a $C_0$-semigroup on $X$. Then the set of all distributionally semi-irregular vectors is a $\calt$-invariant set. Similarly,
  the set of all distributionally irregular vectors is a $\calt$-invariant set.
\end{prop}
\begin{proof}
We will show that the set of all distributionally semi-irregular vectors is a $\calt$-invariant set. For the case of the set of all distributionally irregular vectors, the proof is similar. 

Suppose that $x_0\in X$ is a distributionally semi-irregular vector. 
Then there exists a measurable subset $A\subset \Delta$ with $\udens(A)=1$ such that 
  $$\lim_{t\in A,\, |t|\to \infty}\|T_tx_0\|=0,$$ 
  and there exists a measurable subset 
  $B\subset \Delta$ with $\udens(B)=1$ such that 
  \[
  \lim_{t\in B,\, |t|\to\infty}\|T_tx_0\|>0.
  \]
Fix $t_0\in \Delta$, we will show that 
$T_{t_0}x_0$ is a distributionally semi-irregular vector.  
By Lemma~\ref{transitive-invariant-property-of-density}, $\udens(A-t_0)=\udens(A)=1$ and $\udens(B-t_0)=\udens(B)=1$. It is clear that 
  \[
    \lim_{t\in A-t_0, \, |t|\to \infty}\|T_t(T_{t_0}x_0)\|=\lim_{t\in A, \, |t|\to \infty}\|T_tx_0\|=0
  \]
  and
  \[
    \lim_{t\in B-t_0, \, |t|\to \infty}\|T_t T_{t_0} x_0\|=\lim_{t\in B, \, |t|\to \infty}\|T_tx_0\|>0,
  \]
  which implies that $T_{t_0}x_0$ is a distributionally semi-irregular vector. 
\end{proof}

\begin{thm}\label{DC-point-D-sensitive}
  Let $\calt=\{T_t\}_{t \in \Delta}$ be a $C_0$-semigroup on $X$. If $\{T_t\}_{t \in \Delta}$ admits a distributionally semi-irregular vector, then $\{T_t\}_{t \in \Delta}$ is distributionally sensitive.
\end{thm}
\begin{proof}
  Suppose that $x_0\in X$ is a distributionally semi-irregular vector. 
 Then for any $\eps>0$, there exists some $t_1\in \Delta$ such that $\|T_{t_1}x_0\|<\eps$. By Proposition~\ref{the-orbit-of-DC-vector-is-DC-vector}, 
 $T_{t_1}x_0$ is also distributionally semi-irregular,
 which means that there exists a measurable subset $A\subset \Delta$ with $\udens(A)=1$ such that $\lim_{t\in A,\, |t|\to \infty}\|T_t(T_{t_1}x_0)\|=0$ and there exists a measurable subset 
  $B\subset \Delta$ with $\udens(B)=1$ such that 
  \[
  2\delta:=\lim_{t\in B, \, |t|\to \infty}\|T_t(T_{t_1}x_0)\|>0.
  \]
 Without loss of generality, we can assume that for every $t\in B$, $\|T_t(T_{t_1}x_0)\|\geq \delta$. 
 Thus, for any $x\in X$,
  \[
    B\subset \{t\in \Delta: \|T_t(T_{t_1}x_0)\|\geq \delta \}=
  \{t\in \Delta: \|T_t x- T_t(x+T_{t_1}x_0)\|\geq \delta \}.
  \] Then
\[
  \udens(\{t\in \Delta: \|T_t x- T_t(x+T_{t_1}x_0)\|\geq \delta \})\geq  \udens(B)=1.
\]
Since $x+T_{t_1}x_0 \in B_{\|\cdot\|}(x,\eps)$ and
by the arbitrariness of $\eps$, $\calt$ is distributionally sensitive.
\end{proof}

\subsection{Distributional chaos}
In this subsection we will characterize 
distributional chaos by distributionally semi-irregular and 
distributionally irregular vectors. Then we 
will further investigate the relation of 
distributional chaos and distributionally sensitive. 

First, we introduce the definitions of distributional chaotic and dense distributional chaotic.

\begin{defn}
  Let $\{T_t\}_{t \in \Delta}$ be a $C_0$-semigroup on $X$. A pair $(x,y)\in X\times X $ is called \emph{distributionally chaotic}  if there exists $\delta>0$ such that for any $\eps>0$, we have
  \[
    \udens(\{t\in \Delta: \|T_tx-T_ty\|<\eps\})=1,
  \]
  and
  \[
    \udens(\{t\in \Delta: \|T_tx-T_ty\|> \delta \})=1.
  \]
  
  A $C_0$-semigroup $\{T_t\}_{t \in \Delta}$ is called \emph{distributionally chaotic} if there exists an uncountable  set $K$, such that any two distinct points $x,y$ of $K$ form a distributional chaotic pair, and the set $K$ is called \emph{distributionally scrambled set}. 
  
  A $C_0$-semigroup is called \emph{densely distributionally chaotic} if $\{T_t\}_{t \in \Delta}$ admits a dense distributional scrambled set.
\end{defn}

\begin{rem}\label{rem:d-chaos-d-irregular}
From the definition we can see that $(x,y)\in X\times X $ is a distributionally chaotic pair  if and only if $x-y$ is a distributionally semi-irregular. Thus if the semigroup $\calt:=\{T_t\}_{t \in \Delta}$ is distributionally chaotic, then there exists a distributionally semi-irregular vector.
\end{rem}

Now we give a characterization of the existence of a dense set of distributionally semi-irregular vectors for $C_0$-semigroups on complex sectors.

\begin{thm}
  \label{characterize-of-dense-distributionally-Li-Yorke-chaos}
  Let $\calt=\{T_t\}_{t \in \Delta}$ be a $C_0$-semigroup on $X$. Then the following assertions are equivalent:
\begin{enumerate}
\item $\calt$ admits a dense set of distributionally semi-irregular vectors;
\item $\calt$ admits a residual set of distributionally irregular vectors;
\item $\dprox(\calt)$ is dense in $X\times X$ and $\calt$ is distributionally sensitive;
\item for every sequence $(O_j)_j$ of nonempty open subsets of $X$, there exists a sequence $(K_j)_j$ of Cantor sets with $K_j\subset O_j$ such that $\bigcup_{j=1}^{\infty}K_j$ is a distributionally scrambled set.
  \end{enumerate}
In particular, if in addition $X$ is separable, the above assertions are also equivalent to the following:
  \begin{enumerate}
    \item [(5)] $\calt$ is dense distributional chaos.
  \end{enumerate}
\end{thm}
\begin{proof}

$(2)\Rightarrow(1)$ is trivial.

  $(1)\Rightarrow(3)$. By Theorem~\ref{DC-point-D-sensitive}, $\calt$ is distributionally sensitive.
  Now we will show that  $\dprox(\calt)$ is dense in $X\times X$. 
  
  Denote by $K$ the set of distributionally semi-irregular vectors. Then $K$ is dense in $X$. For any nonempty open subsets
  $U, V\subset X$. Fix $x\in U$. It is clear that $x+K$ is also dense in $X$. Thus we can choose $z\in V\cap (x+K)$. It is easy to see that $(x, z)\in\dprox(\calt)$\,,  which yields that the distributionally proximal relation is dense in $X\times X$. 

  $(3)\Rightarrow(2)$. We first show that the set $\dprox(\calt, \vecz)$ is a dense $G_\delta$ set. By Proposition~\ref{G_delta-set}, $\dprox(\calt, \vecz)$ is a $G_\delta$-subset. 
  
  Now we further show that $\dprox(\calt, \vecz)$ is dense in $X$.
  For any nonempty open subset $U\subset X$, there exists an open neighborhood $W$ of $\vecz$ and a nonempty open subset $V\subset U$ such that $V+W\subset U$. 
  Since $ \dprox(\calt)$ is dense in $X\times X$, there exists $(x,y)\in \dprox(\calt)\cap (W\times (V+W))$, then we have $y-x \in \dprox(\calt,\vecz)$ and $y-x\in (V+W)-W=V\subset U$. Therefore $\dprox(\calt,\vecz)$ is dense $G_{\delta}$ in $X$.
  
 Since $\calt$ is distributionally sensitive, by Proposition~\ref{characterized-for-distributionally-sensitive}(3),
   the set 
  \[
  R_2:= \bigl\{
  y\in X\colon \exists B\subset \Delta  \text{ with }  \udens(B)=1  \text{ s.t } \lim_{t\in B,\, |t|\to\infty} \|T_ty\|=\infty
  \bigr\}
  \]
  is a dense $G_{\delta}$ subset of $X$.
  Then $\dprox(\calt,\vecz)\cap R_2$ is also a dense $G_{\delta}$ set.
  It is clear that for any 
 $x\in \dprox(\calt,\vecz)\cap R_2$, $x$ is a distributionally irregular vector. Thus $\calt$ admits a residual set of distributionally irregular vectors.

  $(3)\Rightarrow(4)$.  Since $\calt$ is distributionally sensitive, by Proposition~\ref{characterized-for-distributionally-sensitive}, $\dunbd(\calt)$ is dense in $X$. Then it is easy to verify that the set
  \[
    R_1= \bigl\{(x,y)\in X\times X\colon \udens(\{t\in \Delta: \|T_tx-T_ty\|> M \})=1, \forall M>0\bigr\}
  \]
  is a dense $G_\delta$ subset of $X\times X$. Thus $R_1\cap \dprox(\calt)$ is also a dense $G_\delta$ set. For any $(x,y)\in R_1\cap \dprox(\calt)$ with $x\neq y$, $(x,y)$ is a distributionally scrambled pair. So the result follows from Theorem \ref{Mycielski}.

  $(4)\Rightarrow(3)$.
  We fix arbitrary nonempty open sets $U_1, U_2 \in X$. Then there exists two Cantor sets $K_1\subset U_1$ and $K_2\subset U_2$ such that $K_1\cup K_2$ is distributionally scrambled. 
  Thus the collection of all distributionally chaotic pairs is dense in $X\times X$. 
  For any $(x,y)\in X\times X$, if $(x,y)$ is a distributionally chaotic pair, then $(x,y)\in\dprox(\calt)$, which implies that $\dprox(\calt)$ is dense in $X\times X$. 
  Besides, if $(x,y)$ is a distributionally chaotic pair, then $x-y$ is a semi-irregular vector. Thus
  $\calt$ is distributionally sensitive by Theorem \ref{DC-point-D-sensitive}. 

Now we obtain that $(1)\Leftrightarrow (2)\Leftrightarrow (3)\Leftrightarrow (4)$, it is easy to verify that if in addition $X$ is separable, $(1)\Leftrightarrow (2)\Leftrightarrow (3)\Leftrightarrow (4)\Leftrightarrow (5)$.
\end{proof}

Applying Theorem~\ref{characterize-of-dense-distributionally-Li-Yorke-chaos}, we have the following characterization of distributional chaos for the $C_0$-semigroup $\{T_t\}_{t \in \Delta}$.

\begin{thm}
\label{characterize-of-distributionally-Li-Yorke-chaos} 
  Let $\calt=\{T_t\}_{t \in \Delta}$ be a $C_0$-semigroup on $X$. Then the following assertions are equivalent:
  \begin{enumerate}
    \item $\calt$ admits a distributionally semi-irregular vector;
    \item $\calt$ admits a distributionally irregular vector;
    \item  $\calt$ is distributionally chaotic;
    \item the restriction of $\calt$ to some closed invariant subspace $\widetilde{X}$ has a dense $G_\delta$ set of distributionally irregular vectors. 
  \end{enumerate}
\end{thm}
\begin{proof}
$(4)\Rightarrow(2)\Rightarrow(1)$ and $(3)\Rightarrow(1)$ are trivial.

$(4)\Rightarrow(3)$. It follows from Theorem \ref{characterize-of-dense-distributionally-Li-Yorke-chaos}.

  $(1)\Rightarrow(4)$. Let $x_0\in X$ be a distributionally semi-irregular vector. Denote $X_0:=\{y\in X \colon \lim_{t\in A,\,|t|\to \infty}\|T_{t}y\|=0
  \}$. Then $x_0\in X_0$ and there exists $A\subset \Delta$ with $\udens(A)=1$ such that $\lim_{t\in A,\, |t|\to \infty}\|T_{t}x_0\|=0$.

For any $t_1\in \Delta$ and $x\in X_0$,
\[
   \lim_{t\in A,\,|t|\to \infty}\|T_{t}T_{t_1}{x}\|=\lim_{t\in A,\,|t|\to \infty}\|T_{t_1}T_{t}{x}\|\leq \|T_{t_1}\|\lim_{t\in A,\,|t|\to \infty}\|T_{t}{x}\|=0.
\]
Thus $X_0$ is a $\calt$-invariant set.

For any $x_1, x_2\in X_0$ and $\alpha, \beta \in \bbk$, 
\[
  \lim_{t\in A,\, |t|\to \infty}\|T_{t}(\alpha x_1+ \beta x_2)\|\leq |\alpha|\lim_{t\in A,\,|t|\to \infty}\|T_{t}{x_1}\|+|\beta|\lim_{t\in A,\,|t|\to \infty}\|T_{t}{x_2}\|=0. 
\]
Thus $X_0$ is a $\calt$-invariant subspace.
Let $\widetilde{X}$ be the closure of $X_0$ in $X$. Then $\widetilde{X}$ is also a $\calt$-invariant subspace of $X$. Since $x_0\in \widetilde{X}$, by Theorem~\ref{DC-point-D-sensitive}, $(\widetilde{X}, \calt)$ is distributionally sensitive. 
It is clear that $X_0\subset \dprox(\calt|_{\widetilde{X}},\vecz)$. Thus $\dprox(\calt|_{\widetilde{X}},\vecz)$ is dense in $\widetilde{X}$ and $\dprox(\calt|_{\widetilde{X}})$ is dense in $\widetilde{X}\times \widetilde{X}$. 
Now $(4)$ follows from Theorem~\ref{characterize-of-dense-distributionally-Li-Yorke-chaos}.
\end{proof}

By the proof of $(1)\Rightarrow (4)$ in Theorem~\ref{characterize-of-distributionally-Li-Yorke-chaos}, we have the following Corollary:

\begin{coro}
Let $\calt=\{T_t\}_{t \in \Delta}$ be a $C_0$-semigroup on $X$. 
Then the set of distributionally irregular vectors is dense in the set of distributionally semi-irregular vectors.
\end{coro}

In the following we will show that 
if $\{T_t\}_{t \in \Delta}$ admits a residual distributionally scrambled set, then actually $X$ is a  distributionally scrambled set.
\begin{thm}
   \label{characterize-of-dense-g-dd-Li-Yorke-chaos}
  Let $\{T_t\}_{t \in \Delta}$ be a $C_0$-semigroup on $X$. Then the following assertions are equivalent:
  \begin{enumerate}
    \item $\{T_t\}_{t \in \Delta}$ admits a residual distributionally scrambled set;
    \item every vector in $X\setminus \{\vecz\}$ is distributionally semi-irregular for $\{T_t\}_{t \in \Delta}$;
    \item  $X$ is a  distributionally scrambled set for $\{T_t\}_{t \in \Delta}$.
  \end{enumerate}
\end{thm}
\begin{proof}
  $(1)\Rightarrow(2)$. Let $K$ be the distributionally scrambled set of $\{T_t\}_{t \in \Delta}$. Then $K$ is residual. For any $x\in X\setminus \{\vecz\}$,  $(K-x)\cap K$ is also residual.
  Take $y\in (K-x)\cap K$. One has $(y,y+x)$ is a distributionally chaotic pair, so as $(e,x)$. Thus
  $x$ is a distributionally semi-irregular vector.
  
  $(2)\Rightarrow(3)$ and $(3)\Rightarrow(1)$ is trivial.
\end{proof}
Let $\calt$ be a $C_0$-semigroup. A vector subspace $Y$ of $X$ is called a distributionally irregular manifold if every vector $y\in Y\setminus\{\vecz\}$ is distributionally irregular for $\calt$.

\begin{thm}\label{thm:f-dense-dis-irr-manifold-1}
Let $X$ be a separable Banach space and  $\calt:=\{T_t\}_{t\in \Delta}$ be a $C_0$-semigroup. 
If $\dasym(\calt,\vecz)$ is dense in $X$, then 
$\calt$ is distributionally sensitive if and only if $\calt$ has a dense distributionally irregular manifold.
\end{thm}
\begin{proof} 
  $(\Leftarrow)$
 It follows from Theorem~\ref{DC-point-D-sensitive}.
  
   $(\Rightarrow)$ If $\calt$ is distributionally sensitive, then by Proposition~\ref{characterized-for-distributionally-sensitive}, the set
    \[
         D:=\Bigl\{x\in X\colon  \exists B\subset \Delta \text{ with } \udens(B)=1 \text{ s.t. } \lim_{t\in B,\, |t|\to \infty} \|T_t x\|=\infty\Bigr\}
    \]
    is a dense $G_\delta$ subset of $X$. First, we need the following claim.
\smallskip 

    \noindent\textbf{Claim}: For any set $A \subset \Delta$ with $\udens(A)=1$, the set
    \[
    P(A)=\Bigl\{x\in X \colon \exists A'\subset A \text{ with } \udens(A')=1 \text{ s.t. } \lim_{t\in A',\, |t|\to\infty} \|T_{t}x\|=0\Bigr\}
    \]
    is a dense $G_\delta$ subset of $X$.
    \begin{proof}[Proof of the Claim]
    Since $\asym(\calt,\vecz)\subset P(A)$, $P(A)$ is dense in $X$. Note that 
    \begin{align*}
    P(A)=\bigcap_{n=1}^\infty
    \bigl\{x\in X\colon  \exists r>n \text{ s.t. } 
		 \mu(\{t\in \Delta_r \colon  \|T_t x\|<\tfrac{1}{n}\}\cap A
   )> \mu(\Delta_r) (1-\tfrac{1}{n})\bigr\}.
   \end{align*}
    Thus, $P(A)$ is a $G_\delta$ subset of $X$.
    \end{proof}

    Fix a dense sequence $(y_i)_i$ in $X$. We will recursively construct  a sequence $(x_i)_i$ in $X$
    such that $\|x_i-y_i\|<\frac{1}{i}$ and $\sspan\{x_i\colon i\in\bbn\}$ is a distributionally irregular manifold.

    By the completeness of $X$, $X_1:=P(\Delta)\cap D$ is a dense $G_\delta$ subset of $X$. We pick $x_1\in X_1$ with $d(x_1,y_1)<1$.
    Then, there exist two subset $A_{1,1}\subset \Delta$
    and $B_{1}\subset \Delta$ with $\udens(A_{1,1})=1$ and $\udens(B_{1})=1$ such that
    \[
    \lim_{t\in A_{1,1},\, |t|\to\infty} \|T_t x_1 \|=0
    \text{ and }
    \lim_{t\in B_{1}, \,|t| \to\infty} \| T_t x_1 \|=\infty.
    \]
    It is clear that $\bbk x_1$ is a distributionally irregular manifold.

 Assume that $x_i\in X$, sets $A_{i,j}\subset \Delta$ and $B_{i}\subset \Delta$ have been constructed for $i=1,2,\dotsc,n$ and $j=1,\dotsc,i$ such that 
    \begin{enumerate}
    	\item $d(x_i,y_i)<\frac{1}{i}$ for $i=1,2,\dotsc,n$;
    	\item for $i=2,\dotsc,n$ and $j=1,\dotsc,i-1$, $A_{i,j}$ is a subset of $A_{i-1,j}$,
    	and $A_{i,i}$ is a subset of $B_{i-1}$;
    	\item for $i=2,\dotsc,n$ and $j=1,\dotsc,i$, 
    	$\udens(A_{i,j})=1$, $\udens(B_{i})=1$, 
    	 \[
    	\lim_{t\in A_{i,j},\, |t|\to\infty} \|T_t x_i \|=0
    	\text{ and }
    	\lim_{t\in B_{i},\, |t| \to\infty} \| T_t x_i \|=\infty.
    	\]
     	
    	\item $\sspan\{x_1,x_2,\dotsc,x_n\}$ is an distributionally irregular manifold.
    \end{enumerate}
By the Claim, 
    \[
    X_{n+1}:= \bigcap_{j=1}^nP(A_{n,j}) \cap P(B_{n})\cap D
    \]
    is also a dense $G_\delta$ subset of $X$.
    Now take $x_{n+1}\in X_{n+1}$ with $d(x_{n+1},y_{n+1})<\frac{1}{n+1}$.
    For $j=1,\dotsc,n$, there exists a subset $A_{n+1,j}$ of $A_{n,j}$ 
    such that 
    \[ 
    \lim_{t\in A_{n+1,j}, \, |t|\to\infty} \|T_t {x_{n+1}}\|=0,
    \]
    a  subset $A_{n+1,n+1}$ of $B_{n}$
    such that 
    \[ 
    \lim_{t\in A_{n+1,n+1}, \, |t|\to\infty}\|T_t x_{n+1}\|=0,
    \]
    and a subset $B_{n+1}\subset \Delta$ such that
    \[ 
    \lim_{t\in B_{n+1} ,\, |t|\to\infty}\|T_t {x_{n+1}}\|=\infty.
    \]
    For any $\sum_{\ell=1}^{n+1} \alpha_\ell x_\ell\in  \sspan\{x_1,\dotsc,x_n,x_{n+1}\}\setminus\{\vecz\}$, we have
    \begin{align*}
    \lim_{t\in A_{n+1,1},\, |t|\to\infty} \left\|T_t
    \biggl(\sum_{\ell=1}^{n+1} \alpha_\ell x_\ell\biggl)\right\|&
    \leq \sum_{\ell=1}^{n+1}|\alpha_\ell|\lim_{t\in A_{n+1,1}, \, |t|\to\infty} \|T_t x_{\ell}\|
    =0.
    \end{align*}
    Let $i=\min\{\ell\in \{1,\dotsc,n+1\}\colon \alpha_\ell\neq 0\}$.
    If $i<n+1$, then
    \begin{align*}
    \lim_{t\in A_{n+1, i+1},\, |t|\to\infty} \|T_t (\sum_{\ell=i}^{n+1} \alpha_\ell x_\ell)\|
    &\geq 
    |\alpha_i| 
    \lim_{t\in A_{n+1, i+1},\, |t|\to\infty}\|T_t x_i\| -\sum_{\ell=i+1}^{n+1} |\alpha_\ell|\limsup_{t\in A_{n+1, i+1}, \, |t|\to\infty} \|T_t x_\ell\| \\
    &=\infty-0=\infty.
    \end{align*}
    If $i=n+1$, then by the construction one has 
    \[
    \lim_{t\in B_{n+1} ,\, |t| \to\infty} \| T_{t}(x_{n+1})\|=\infty.
    \]
    So $\sspan\{x_1,\dotsc,x_n,x_{n+1}\}$ is a distributionally irregular manifold. A recursive argument gives us a sequence $(x_i)_i$ fulfilling that the subspace $\sspan\{x_i\colon i\in\bbn\}$ is a dense distributionally irregular manifold.
\end{proof}
    
\section{Distributionally chaos for translation \texorpdfstring{$C_0$}{C0}-semigroups on complex sectors}

In this section, 
we will investigate distributional chaos for translation $C_0$-semigroups. First, we introduce the translation $C_0$-semigroups, which has been generalized to complex sectors in \cite{CP2009}.

\begin{defn}\label{defn:admissible}
Let $\Delta$ be a complex sector. 
A measurable function $v:\Delta\to \bbr$ is said to be an \emph{admissible weight function} if it verifies $v(t)>0$ for every $t\in\Delta$, and there exist constants $M\geq 1$ and $w\in \bbr$ such that  
\[v(t)\leq M e^{w|t'|}v(t+t')\ \text{for all} \ t,t'\in \Delta.\]
\end{defn}

Now, we can define the spaces where the translation $C_0$-semigroup will act.

\begin{defn}
Let $v:\Delta\to \bbr$ be an admissible weight function.
For $1\leq p<\infty$ and $\bbk=\bbr$ or $\bbc$, we define the space 
\[L_v^p(\Delta):=\{f: \Delta \to \bbk: f  \text{ is measurable and } \|f\|\leq \infty\},
\]
with $\|f\|:=\Bigl(\int_{\Delta} |f(t)|^{p}v(t)\dd t\Bigr)^{\frac{1}{p}}<\infty$. 
\end{defn}
    
\begin{defn}
Let $X:=L_v^p(\Delta)$. 
For $t\in\Delta$ and $f\in X$, we define $T_t f$ as $T_t(f(s)):=f(s+t)$ for every  $s\in\Delta$. 
We call $\{T_t\}_{t\in \Delta}$ the \emph{translation semigroup} on $X$.
\end{defn}

The second condition in Definition \ref{defn:admissible}, as in the real case, allows the translation
semigroup to be $C_0$-semigroup on our spaces
(see Example $7.4$ of \cite{GP2011} for the real case).

First, we will characterizes distributional chaos for the translation semigroup on complex sectors.

\begin{thm}\label{characterization-of-DC-chao-for-translation-semigroups}
  Let $v$ be an admissible weight function and $\calt=\{T_t\}_{t\in \Delta}$ be the translation semigroup on $X=L^p_v (\Delta)$, Then the following assertions are equivalent:
  \begin{enumerate}
    \item there exist some $f\in X$ and $\delta>0$ such that $\udens(\{t\in \Delta: \|T_t f\|\geq \delta\})=1$;
    \item $\calt$ is distributionally sensitive;
    \item $\calt$ is densely distributionally chaotic;
    \item $\calt$ admits a dense distributionally irregular manifold.
  \end{enumerate}
\end{thm}
\begin{proof}
$(1)\Rightarrow(2).$ Assume that  there exists $f\in X$ and $\delta>0$ such that $\udens(\{t\in \Delta: \|T_t f\|\geq \delta\})=1$. 
\medskip
Let $A:=\{t\in \Delta: \|T_t f\|\geq \delta\}$. 
Since $f\in X:=L^p_v (\Delta)$,
for any $k>0$, there exists
$m_k\in \bbn$ such that
$\int_{\Delta\setminus \Delta_{m_k}}|f(t)|^p\rho(t)\dd t<\frac{1}{k^p}$.
Define 
\[
h(t)= \begin{cases}
f(t),  &  \ \Delta\setminus \Delta_{m_k}; \\
0,  & \ \Delta_{m_k}. \\
\end{cases}
\]
then 
\[\|h\|=\biggl(\int_{\Delta\setminus \Delta_{m_k}}|f(t)|^p\rho(t)\dd t\biggr)^{\frac{1}{p}}<\frac{1}{k}.
\]
For any $s\in \Delta\setminus\Delta_{m_k}$, 
$s+\Delta\subset \Delta\setminus\Delta_{m_k}$, then
for any $t\in\Delta$ we have
\[T_sh(t)=h(s+t)=f(s+t)=T_sf(t).\]
Now we define $B:=A\cap(\Delta\setminus\Delta_{m_k})$, then $\udens(B)=1$ since $\udens(A)=1$.  
For any $g\in X$, let $g_1:=g+h$ then $\|g-g_1\|=\|h\|<\frac{1}{k}$, and for any $s\in B$
\[\|T_sg-T_sg_1\|=\|T_s h\|=\|T_s f\|\geq \delta,
\]
which implies that $\udens\{t\in \Delta: \|T_sg-T_sg_1\|\geq \delta\}=1$. Hence, 
$\{T_t\}_{t\in \Delta}$ is distributionally sensitive. 

   $(2)\Rightarrow(4)$.  
   Let $X_0$ be the set of all functions in $X$ with compact support.
   Then $X_0$ is dense in $X$ and $X_0\times X_0$ is dense in $X\times X$. Now we show that $X_0\times X_0\subset\dasym(\calt)$.
   For any $f,g\in X_0$, since $f,g$ are functions with compact support, there exists $r_0$ such that when $|s|>r_0$, $\|T_s f\|=\|T_s g\|=0$.
   Thus for any $\eps>0$, $\dens\{t\in\Delta:\|T_t f-T_t g\|<\eps\}=1$, which implies that $(f,g)\in \dasym(\calt)$. 
   Then $\dasym(\calt)$ is dense in $X\times X$.
  Similar to the proof of Theorem \ref{characterize-of-dense-distributionally-Li-Yorke-chaos} $(3)\Rightarrow(1)$, we can show that 
  $\dasym(\calt,\vecz)$ is dense in $X$. Then the result follow from Theorem~\ref{thm:f-dense-dis-irr-manifold-1}.

   $(4)\Rightarrow(3)$. It follows from Theorem \ref{characterize-of-distributionally-Li-Yorke-chaos}.

   $(3)\Rightarrow(1)$ is trivial.
\end{proof}

Now we will give a sufficient condition for dense distributional chaos, expressed in terms of the weight are given.

\begin{thm}\label{necessary-of-distributionally-chaos}
  Let $v$ be an admissible weight function and $\calt=\{T_t\}_{t\in \Delta}$ be the translation $C_0$-semigroup on $X=L^p_v (\Delta(\alpha))$. If there exists a subset $K\subset\bbn$ with $\udens(K)=1$ such that 
  \[
    \sum_{k\in K}\int_{\Delta_{k,k+1}} v(s) \dd s<\infty,
  \] 
  where $\Delta_{k,k+1}=\{t\in \Delta: k \leq |t|\leq k+1\}$,
  then the translation $C_0$-semigroup $\calt=\{T_t\}_{t\in \Delta}$ is dense distributionally chaotic. 
\end{thm}
\begin{proof}
 
Denote 
$K\cap[1,n]=\{n_1,\dots, n_{k_n}\}$ with $n_{1}<\dotsc<n_{k_n}$.
\begin{align*}
   \udens(\bigcup_{k\in K}\Delta_{k-1,k})&=
  \limsup_{r\to \infty}\frac{\mu(\bigcup_{k\in K}\Delta_{k-1,k}\cap \Delta_r)}{\mu(\Delta_{r})}
  \geq
  \limsup_{n\to \infty}\frac{\mu(\bigcup_{k\in K\cap[1,n]}\Delta_{k-1,k})}{\mu(\Delta_{n_{k_n}})}\\
  &\geq
  \limsup_{n\to \infty}
  \frac{\alpha(\Sigma_{k\in K\cap[1,n]}(2k-1))}{\alpha n^2}
  \geq
\limsup_{n\to \infty}
  \frac{\alpha(2\Sigma_{i=1}^{k_n}i-k_n)}{\alpha n^2}\\
  &=\limsup_{n\to \infty} \frac{k^2_n}{n^2}
  =\biggl(\limsup_{n\to \infty} \frac{\#(K\cap[1,n])}{n}\biggr)^2
  =1.
\end{align*}

Let 
\[
f(s) = 
\begin{cases}
1 & \text{ } s\in \bigcup_{k\in K}\Delta_{k,k+1}; \\
0 & \text{ others }.
\end{cases}
\]
Since $\|f\|=\int_{\Delta} |f(s)|^pv(s)\dd s=\sum_{k\in K}\int_{\Delta_{k,k+1}} v(s) \dd s<\infty$, $f\in L^p_v (\Delta(\alpha))$. 
By the admissibility
of the weight, 
pick $\delta>0$ such that $v(s)\geq \frac{\delta^p}{\alpha}$ for every $s\in \overline{\Delta_2}$.
For any $t\in \bigcup_{k\in K}\Delta_{k-1,k}$, 
there exists $k\in K$ such that
$t\in \Delta_{k-1,k}$. Then $\mu((t+\overline{\Delta_2})\cap\Delta_{k,k+1})\geq\mu(\Delta_1)=\alpha$.
Thus,
\begin{align*}
  \|T_t f\|&\geq \biggl(\int_{\overline{\Delta_2}}|T_t f(s)|^p v(s)\dd s\biggr)^{\frac{1}{p}}
     \geq  \delta  \biggl( \frac{1}{\alpha} \int_{\overline{\Delta_2}}| f(s+t)|^p \dd s\biggr)^{\frac{1}{p}}\\
           &=\delta \biggl(\frac{1}{\alpha} \int_{t+\overline{\Delta_2}} |f(s)|^p \dd s\biggr)^{\frac{1}{p}}
           \geq\delta \biggl(\frac{1}{\alpha} \int_{(t+\overline{\Delta_2})\cap\Delta_{k,k+1}} |f(s)|^p \dd s\biggr)^{\frac{1}{p}}\\ 
           &=\delta \biggl(\frac{1}{\alpha} \mu((t+\overline{\Delta_2})\cap\Delta_{k,k+1})\biggr)^{\frac{1}{p}}\geq\delta.
\end{align*}

Then $\{t\in \Delta: \|T_t f\|\geq \delta\}\supset \bigcup_{k\in K}\Delta_{k-1,k}$ and 
$\udens(\{t\in \Delta: \|T_t f\|\geq \delta\})\geq \udens(\bigcup_{k\in K}\Delta_{k-1,k})=1$, and the result follows from Theorem \ref{characterization-of-DC-chao-for-translation-semigroups}. 
\end{proof}

Using Theorem \ref{necessary-of-distributionally-chaos} we provide two examples of dense distributional chaotic for the transition $C_0$-semigroups on complex sectors.

\begin{exam}
  Let  $1\leq p<\infty$. Consider the translation $C_0$-semigroup $(T_{t})_{t\in \Delta(\frac{\pi}{4})}$ defined on $X=L^p_{v}(\Delta(\frac{\pi}{4}))$ with $v(t)=e^{-|t|}$, $t\in \Delta(\frac{\pi}{4}) $. Then, for any $t_1\in \Delta(\frac{\pi}{4})$ and $t_2\in \Delta(\frac{\pi}{4})$, we have 
  \[
    e^{|t_2|}v(t_1+t_2)=e^{|t_2|}e^{-|t_1+t_2|}\geq e^{-|t_1|}=v(t_1).
  \]
   Thus $v$ satisfies the admissible weight condition (see Definition \ref{defn:admissible}) and
   $\{T_t\}_{t\in \Delta(\frac{\pi}{4})}$ is a $C_0$-semigroup on $X=L^p_{v}(\Delta(\frac{\pi}{4}))$.
  Since
  \[
    \int_{\Delta({\frac{\pi}{4}})} v(t)\dd t=\frac{\pi}{2}<\infty,
  \]
  the translation $C_0$-semigroup is densely distributional chaotic by Theorem~\ref{necessary-of-distributionally-chaos}.
\end{exam}
\begin{exam}
Let $1\leq p<\infty$. Consider the translation semigroup $(T_{t})_{t\in \Delta(\frac{\pi}{4})}$ defined on $X=L^p_{v}(\Delta(\frac{\pi}{4}))$ with $v(t):=\frac{1}{|t|^4+1}$, $t\in \Delta(\frac{\pi}{4})$.

One can easily verify that $v$ satisfies the admissible weight condition (see Definition \ref{defn:admissible}) and thus
$\{T_t\}_{t\in \Delta(\frac{\pi}{4})}$ is a $C_0$-semigroup on $X=L^p_{v}(\Delta(\frac{\pi}{4}))$.
Since
  \[
    \int_{\Delta(\frac{\pi}{4})} v(t)\dd t=\frac{\pi^2}{8}<\infty
  \]
the translation $C_0$-semigroup is densely distributional chaotic by Theorem~\ref{necessary-of-distributionally-chaos}.
\end{exam}

Let $\calt=\{T_t\}_{t \in \Delta}$ be a $C_0$-semigroup on $X$. Recall that the $C_0$-semigroup $\{T_t\}_{t\in \Delta(\alpha)}$ is called \emph{hypercyclic} if there exists $x\in X$ such that the orbit $\{T_t x\colon t\in \Delta\}$ is dense in $X$. If in addition the $C_0$-semigroup $\{T_t\}_{t\in \Delta(\alpha)}$ admits a dense set of periodic points, then it is called \emph{Devaney chaos}.

In {\cite[Corollary 1]{CP2007}}, Conejero and Peris gave the following sufficient condition for a translation semigroup to be Devaney chaos. 

\begin{prop}\label{character-of-Devaney-chaos}
  Let $X=L^p_v(\Delta)$ and let $R\subset \Delta$ be a ray that it is not contained in the boundary of $\Delta$. The following are equivalent:
  \begin{enumerate}
  \item There exists $t\in R$ such that 
  \[
    \sum_{k=0}^{\infty} v(kt)< \infty;
  \]
  \item The restriction $\{T_t\}_{t\in R}$ of the translation semigroup to the ray $R$ admits a non-trivial periodic point.
  \end{enumerate}
  Besides, any of these conditions implies that the translation semigroup $\{T_t\}_{t\in \Delta}$ on $X$ is Devaney chaotic.
\end{prop}

In \cite[Corollary 2.6]{BP2012}, Barrachina and Peris showed that Devaney chaos implies distributionally chaos for the translation $C_0$-semigroup with $\bbr^+$ as an index set, where $\bbr^+=\{t\in\bbr\colon t\geq 0\}$. 
However, in the following we provide an example to illustrate that for the translation
$C_0$-semigroup $\calt=\{T_t\}_{t \in \Delta}$ on $X=L^p_v (\Delta(\alpha))$, Devaney chaos does not imply distributional chaos.

\begin{exam}\label{Devaney-chaos-not-distributionally-chaos}
  Consider the translation semigroup $\calt=\{T_t\}_{t\in \Delta(\frac{\pi}{4})}$ on $X:=L^p_{v}(\Delta(\frac{\pi}{4}))$ with $v(x+iy) = e^{2y}$. 

For any $t_1:=(x_1+iy_1)\in \Delta(\frac{\pi}{4})$ and $t_2:=(x_2+iy_2)\in \Delta(\frac{\pi}{4})$, we have
  \[
    e^{2|t_2|}v(t_1+t_2)=e^{2\sqrt{x_2^2+y_2^2}}e^{2y_1+2y_2}\geq e^{2y_1}=v(t_1).
  \]
  Thus $v$ satisfies the admissible weight condition (see Definition \ref{defn:admissible}) and
   $\{T_t\}_{t\in \Delta(\frac{\pi}{4})}$ is a $C_0$-semigroup on $X=L^p_{v}(\Delta(\frac{\pi}{4}))$.
  
  Consider the ray $R:=\{r-\frac{r}{2}i\in \Delta(\frac{\pi}{4}):r\in\bbr^+\} \subset \Delta(\frac{\pi}{4})$ and $t_1:=2-i\in R$. Since
  \[
    \sum_{k=0}^{\infty} v(kt_1)= \sum_{k=0}^{\infty} e^{-2k}=\frac{e^2}{e^2-1}<\infty.
  \]
  By Proposition~\ref{character-of-Devaney-chaos}, $\{T_t\}_{t\in \Delta(\frac{\pi}{4})}$ is Devaney chaotic.
  
  For any $f\in L^p_{v}(\Delta(\frac{\pi}{4}))$ and $\eps>0$,
there exists
$r(\eps)\in \bbn$ such that
\[\int_{\Delta\setminus \Delta_{r(\eps)}}|f(t)|^pv(t)\dd t<\eps^p.
\]
Let $A=\{x+iy\in \Delta(\frac{\pi}{4}): y \geq 0\}$. For any $t\in (\Delta\setminus \Delta_{r(\eps)})\cap A$, $t+\Delta \subset \Delta\setminus \Delta_{r(\eps)}$ and $v(s)\leq v(s+t)$ for all $s\in\Delta$, then we have 
  \begin{align*}
  \|T_{t} f\|&=
  \biggl(\int_\Delta |f(s+t)|^p v(s)\dd s\biggr)^{\frac{1}{p}}\leq  \biggl(\int_{\Delta} |f(s+t)|^p v(s+t)\dd s\biggr)^{\frac{1}{p}}\\
 &=\biggl(\int_{t+\Delta} |f(s)|^p v(s)\dd s\biggr)^{\frac{1}{p}}\leq \biggl(\int_{\Delta\setminus \Delta_{r(\eps)}}|f(s)|^pv(s)\dd s\biggr)^{\frac{1}{p}}< \eps.
\end{align*}
Thus $\{t\in \Delta: \|T_t f\|< \eps\}\supset(\Delta\setminus \Delta_{r(\eps)})\cap A$.
This implies that $\ldens\{t\in \Delta: \|T_t f\|< \eps\}\geq
\ldens((\Delta\setminus \Delta_{r(\eps)})\cap A)=\frac{1}{2}$. Therefore 
\[\udens(\{t\in \Delta: \|T_t f\|\geq \eps\})=1-\ldens(\{t\in \Delta: \|T_t f\|<\eps\})\leq 1-\frac{1}{2}=\frac{1}{2}.\]
By Theorem \ref{characterization-of-DC-chao-for-translation-semigroups}, $\{T_t\}_{t\in \Delta(\frac{\pi}{4})}$ is not distributionally chaotic. 
\end{exam}

In fact, in Example \ref{Devaney-chaos-not-distributionally-chaos} we can further show that the restriction $\{T_t\}_{t\in R}$ of the translation semigroup $\calt=\{T_t\}_{t\in \Delta(\frac{\pi}{4})}$ to the ray $R:=\{r-\frac{r}{2}i\in \Delta(\frac{\pi}{4}):r\in\bbr^+\} \subset \Delta(\frac{\pi}{4})$ is distributionally chaotic. 
By Proposition~\ref{character-of-Devaney-chaos}, there exists $t_1\in R$ and $f_1\in L^p_{v}(\Delta(\frac{\pi}{4}))$ such that $T_{t_1} f_1=f_1$. Let $X_0$ be the dense subspace of continuous functions with compact supports. 
It is easy to check that $\lim_{n\to \infty} T_{t_1}^n f=0$ for any $f\in X_0$. Therefore, by \cite[Corollary 35]{BBPW2018} $T_{t_1}$ is distributional chaotic. Then by \cite[Theorem 5]{BBPW2018} $\{T_{t}\}_{t\in R}$ is distributionally chaotic.

\bigskip 

\noindent \textbf{Acknowledgment.}
J. Li was supported in part by NSF of China (12222110). Y. Yang was partially supported by STU Scientific Research Initiation Grant (SRIG, No. NTF24025T).


\end{document}